\documentclass[12pt]{article}

\usepackage{amssymb,amsthm,amsmath}

\usepackage{url}

\theoremstyle{plain}

\newtheorem{theorem}{Theorem}[section]
\newtheorem{corollary}[theorem]{Corollary}

\newtheorem{proposition}[theorem]{Proposition}
\newtheorem{prob}[theorem]{Problem}

\theoremstyle{definition}

\newtheorem{remark}[theorem]{Remark}

\newcommand{\R}{{\mathbb R}}

\newcommand{\eps}{\varepsilon}

\DeclareMathOperator{\card}{card}

\newcommand{\PP}{\mathbb{P}} 
\newcommand{\EE}{\mathbb{E}} 

\begin{document}  

\title{Khinchine inequality for slightly dependent random variables}

\author{ Susanna Spektor}
\date{}

\newcommand\address{\noindent\leavevmode

\noindent
Susanna Spektor, \\
Dept.~of Math.~and Stat.~Sciences,\\
University of Alberta, \\
Edmonton, Alberta, Canada, T6G 2G1.\\
E-mails: \texttt{\small sanaspek@gmail.com}

}


\maketitle

\begin{abstract}
We prove a Khintchine type inequality under the  assumption that the sum of
Rademacher random variables equals zero. As an application we show a new
tail-bound for a hypergeometric random variable.

\medskip

\noindent 2010 Classification: 46B06, 60E15
%

\noindent Keywords:
Khintchine inequality, Kahane inequality, Rademacher random variables,
Hypergeometric distribution.
\end{abstract}




\setcounter{page}{1}

\section{Introduction}

The Khinchine inequality plays a crucial role in many deep results of Probability and Analysis
(see \cite{Garlin, Kah, LT, MS, PShir, NTJ} among others). It says that  $L_p$ and $L_2$ norms  of sums of weighted
independent Rademacher random variables are comparible.
More precisely, we say that $\eps _0$ is a Rademacher random variable if
$\PP(\varepsilon_0=1)=\PP(\varepsilon_0=-1)=\displaystyle{\tfrac 12}$.
Let $\varepsilon_i$, $i\leq N$,  be independent
copies of $\varepsilon_0$ and $a \in \R^{N}$.
The Khinchine inequality
(see e.g. Theorem 2.b.3 in \cite{LT}  or Theorem 12.3.1 in \cite{Garlin})
states that for any $p\geq 2$ one has
\begin{align}\label{1}
 \left(\EE\left|\sum_{i=1}^{N}a_i\varepsilon_i\right|^p\right)^{\frac 1p}\leq
 \sqrt p \, \|a\|_{2}=\sqrt p \left(\EE\left|\sum_{i=1}^{N}a_i\varepsilon_i\right|^2\right)^{\frac 12}.
\end{align}

Note that the (Rademacher) random vector $\varepsilon = (\varepsilon_1, \ldots, \varepsilon_{N})$
in the Khinchine inequality has independent coordinates. However in many problems of Analysis and
Probability it is important to consider random vectors with dependent coordinates, e.g.  so-called
log-concave random vectors, which in general have dependent coordinates, but
 whose behaviour is similar to that of Rademacher random vector or to the Gaussian random vector
(see e.g. \cite{G} and references there in). In \cite{O'Rurke} the author considered random
matrices, whose rows are independent random vectors
satisfying certain conditions (so the vectors may have dependent coordinates).
He studied limiting empirical distribution of eigenvalues of
such matrices. As an example of such a vector, showing that the conditions cover large
class of natural distributions, not covered by previously known results,
O'Rourke considered the vector
$\varepsilon = (\varepsilon_1, \ldots, \varepsilon_{N})$,
whose coordinates are Rademacher random variables under the additional condition
\begin{align}\label{2}
  S=\sum_{i=1}^{N}\varepsilon_i=0
\end{align}
(see Examples~1.4 and 1.10 in \cite{O'Rurke}).
For such vectors he proved a Khintchine type inequality with the factor $C\sqrt{N} p /\log N$
in front of $\|a\|_2$, which was enough for his purposes. The goal of this paper is to show that
such random variables satisfy a Khintchine type inequality with the same factor $\sqrt{p}$
as in the standard Khintchine  inequality.  To shorten notation, by $\EE_S$ we  denote an
expectation with assumption (\ref{2}). Note that the corresponding probability space is
\begin{align}\label{set}
 \Omega=\left\{\varepsilon \in\{-1, 1\}^{N}\, |\,
 \sum_{i=1}^{N}\varepsilon_i=0\right\}=\left\{\varepsilon\in\{-1, 1\}^{N}\, |\,
  \card\{i: \, \varepsilon_i=1\}=n\right\}.
\end{align}

Our main result is the following theorem.

\begin{theorem}\label{mainthm}
Let $\varepsilon_i, i\leq N$,  be Rademacher random variables satisfying condition (\ref{2}).
Let $a =(a_1, \ldots, a_N) \in \R^{N}$ and $\displaystyle{b=\tfrac{1}{N}\sum_{i=1}^{N}a_i}$. Then
\begin{equation}\label{main}
 \left(\EE_S\left|\sum_{i=1}^{N}a_i\varepsilon_i\right|^p\right)^{1/p} \leq
  \sqrt{ 2p } \, \left(\|a\|_{2}^2- N\, b^2 \right)^{1/2}
  \leq
 \sqrt{ 2p } \, \left(\EE_S\left|\sum_{i=1}^{N}a_i\varepsilon_i\right|^2\right)^{1/2}.
\end{equation}
\end{theorem}

The first step in the proof is a reformulation in terms of random variables on the
permutation group as follows. Let $N=2n$. For the set $\Omega$ defined in (\ref{set}),
we put into correspondence the group $\Pi_{N}$ of all permutations of the set $\{1,..., N\}$ as
\begin{align*}
   \sigma \in \Pi_{N}\longleftrightarrow A_{\sigma}=\left\{\varepsilon \in \Omega \,\,
    | \,\, \varepsilon_i=1 \,\, \mbox{if} \,\,\sigma(i)\leq n; \, \varepsilon_i=-1
   \,\,\mbox{if}\,\,\sigma(i)>n\right\}.
\end{align*}
Given $a \in \R^{N}$, define $f_a :\Pi_{N}\longrightarrow\R$ by
\begin{align}\label{function}
    f_a(\sigma) :=\left|\sum_{i=1}^na_{\sigma(i)}-\sum_{i=n+1}^{2n}a_{\sigma(i)}\right|.
\end{align}
By $\EE_{\Pi}$ we denote the average over $\Pi_{N}$, i.e. the expectation with respect
to the normalized counting measure on $\Pi_{N}$. Note, that $\EE_S\left|\sum_{i=1}^{N}
a_i \varepsilon_i\right|^p=\EE_{\Pi}|f|^p$. Therefore Theorem~\ref{mainthm} is equivalent to
the following theorem.

\begin{theorem}\label{proppsi2}
Let $N=2n$, $a \in \R^{N}$. Let $f_a$ be the function defined in (\ref{function}).
Let $\displaystyle{b=\tfrac{1}{N}\sum_{i=1}^{N}a_i}$. Then, for $p\geq 2$
\begin{align}\label{n9}
 \left(\EE_{\Pi}|f_a|^p\right)^{1/p}\leq \sqrt{2p}\left(\sum_{i=1}^{N}a_i^2- N \,
  b^2\right)^{1/2}\leq \sqrt{2p} \left(\EE_{\Pi}|f_a|^2\right)^{1/2}.
\end{align}
\end{theorem}

In Section~\ref{smain} we prove Theorem~\ref{proppsi2}.
Then,  in Section~\ref{hyp},
we consider  a special case of our problem, when the coordinates of the vector $a$ are either ones or zeros.
This particular case leads to the hypergeometric distribution. We obtain a new bounds for the
$p$-th central moments of such variables.

Finally let us note the setting of Theorem~\ref{proppsi2} can be extended to more general case.
We pose the following problem.

\begin{prob}\label{probl}
 Let $a, b \in \R^N$. Estimate $M(a,b,p)$ as a function of $a,b,p,N, M(a,b,2)$. For example, is it true that
$$
   M(a,b,p)\leq M(a,b,2)+C(a,b,p,N),
$$
or that
$$
M(a,b,p)\leq C(p,N)M(a,b,2)?
$$
\end{prob}
\noindent We discuss this problem in the last Section.



\section{Proof of Theorem~\ref{proppsi2}}
\label{smain}



Direct calculations show that
$$
 \EE_{\Pi}|f|^2=\frac{N \|a\|^2_2-\left(\sum_{i=1}^{N}a_i\right)^2}{N(N-1)}.
$$
Thus, without loss of generality we may assume that $\displaystyle{\sum_{i=1}^{N}a_i=0}$.

For $k\leq n$ denote  $\displaystyle{b_{k, \sigma}:=a_{\sigma(k)}-a_{\sigma(n+k)}}$ and by $\displaystyle{H_{k, \sigma}:=\sum_{i=k+1}^na_{\sigma(i)}-\sum_{i=n+k+1}^{2n}a_{\sigma(i)}}$ (with $H_{n, \sigma}=0$). Clearly,
\begin{align*}
\sum_{i=1}^na_{\sigma(i)}-\sum_{i=n+1}^{2n}a_{\sigma(i)}&=b_{1, \sigma}+H_{1, \sigma}=b_{1, \sigma}+b_{2, \sigma}+H_{2, \sigma}=\ldots=\sum_{i=1}^nb_{i, \sigma}.
\end{align*}
Note, that
$\displaystyle{\EE_{\Pi}\left|b_{1, \sigma}+H_{1, \sigma}\right|^p=\EE_{\Pi}\left|-b_{1, \sigma}+H_{1,
\sigma}\right|^p}$. Hence,
$$
\EE_{\Pi}|f_a(\sigma)|^p=\EE_{\Pi}\left|\sum_{i=1}^na_{\sigma(i)}-\sum_{i=n+1}^{2n}a_{\sigma(i)}\right|^p
=\frac{\EE_{\Pi}\left|b_{1, \sigma}+H_{1, \sigma}\right|^p+\EE_{\Pi}\left|-b_{1,
\sigma}+H_{1, \sigma}\right|^p}{2}.
$$
Thus, denoting by $\delta_i, i\leq n$, i.i.d. Rademacher random variables independent of $\varepsilon_1, \ldots, \varepsilon_N,$ and using Khinchine inequality (\ref{1}), we obtain
\begin{align*}
\EE_{\Pi}|f_a(\sigma)|^p&=\EE_{\Pi}\EE_{\delta_1}\left|\delta_1\,b_{1, \sigma}+H_{1, \sigma}\right|^p\\
&=\EE_{\Pi}\EE_{\delta_1}\EE_{\delta_2}\left|\delta_1\,b_{1, \sigma}+\delta_2\,b_{2, \sigma}+H_{2, \sigma}\right|^p
=\ldots=\EE_{\Pi}\EE_{\delta_1}\EE_{\delta_2}\ldots\EE_{\delta_n}\left|\sum_{i=1}^n\delta_i\,b_{i, \sigma}\right|^p\\
&\leq \EE_{\Pi}\left[\sqrt p \left(\sum_{i=1}^nb_{i, \sigma}^2\right)^{1/2}\right]^p
= p^{p/2}\, \EE_{\Pi}\left(\sum_{i=1}^n\left|a_{\sigma(i)}-a_{\sigma(i+n)}\right|^2\right)^{p/2}\\
&\leq p^{p/2}\, \EE_{\Pi}\left(2\sum_{i=1}^n\left(a_{\sigma(i)}^2+a_{\sigma(i+n)}^2\right)\right)^{p/2}\leq
(2p)^{p/2}\, \|a\|_2^p,
\end{align*}
which completes the proof.
\qed


\bigskip

\section{Hypergeometric distribution}\label{hyp}

In this section we discuss a specific case of hypergeometric distribution  and show how it is related to our problem.
Recall that hypergeometric random variable with parameters $(N, n, \ell)$ is a random variable $\xi$ which takes values  $k=0, \ldots, \ell$ with probability $$
\displaystyle{p_k=\frac{{\ell \choose k}{N-\ell \choose n-k}}{{N \choose n}}}.
$$
In this section we consider only the case $N=2n, \, \ell \leq n$. It is well known that $\displaystyle{\EE \, \xi=\ell/2}$.
In the next proposition we estimate the central moment of $\xi$.

\begin{proposition}\label{hypergeometric}
Let  $1\leq \ell \leq n$. Let $\xi$ be $(2n, n, \ell)$ hypergeometric random variable. Then for  $p \geq 2$ one has
\begin{align*}
\EE \, |\xi-\EE\, \xi|^p\leq \sqrt 2 \left(\frac{p \, \ell}{4}\right)^{\frac p2}.
\end{align*}
\end{proposition}

\bigskip

\begin{remark}
It is well known (see e.g. Theorem~1.1.5 in \cite{Chafai}) that the conclusion of Proposition~\ref{hypergeometric} is equivalent to the following, so-called $\psi_2$ deviation inequality.
\begin{align}\label{11}
\forall t \geq 1 \quad \quad \PP(|\xi-\EE\, \xi|> t)\leq \exp\left(\frac{-c t^2}{\ell}\right),
\end{align}
in the following sense: Proposition 3.1 implies (\ref{11}) for some absolute constant $c>0$ and vise versa, there is another absolute constant $C>0$, such that \ref{11} implies Proposition 3.1.
This estimate, for hypergeometric $\xi$, is of independent interest, in particular it
is better than the previously observed bound $\exp(-2t^2/n)$ when $\ell\ll n$ (see Section 6.5 of [8] and formulas (10), (14) in    \cite{Skala}).
\end{remark}

\bigskip
\begin{remark}
One can use Theorem \ref{proppsi2} to estimate
$\EE_S|\sum_{i=1}^{2n}a_i \varepsilon_i|^p$ in the case that the vector  $a$
has $0/1$ coordinates with $\ell$ ones. Indeed, without loss of generality  assume that $a_1=a_2=\ldots =a_{\ell}=1$
and $a_{\ell+1}=a_{\ell+2}=\ldots =a_{2n}=0$.  Then,
$\sum_{i=1}^{2n}a_i \varepsilon_i=\sum_{i=1}^{\ell}\varepsilon_i$.
Theorem~\ref{proppsi2} implies the following estimate.
\end{remark}
\begin{corollary}\label{corhyp}
Let $a \in \R^{N}$, $N=2n$, be a vector with $\ell$ coordinates equals to one and $N-\ell$ zero coordinates.
Then, for $p\geq 2$,
\begin{align*}
 \EE_S\left|\sum_{i=1}^{N}a_i \varepsilon_i\right|^p\leq (2 \, p\, \ell)^{p/2}.
\end{align*}
\end{corollary}

\bigskip

\noindent
{\bf Proof of Proposition~\ref{hypergeometric}. }
Denote
$X:=\sum_{i=1}^{2n}a_i \varepsilon_i=\sum_{i=1}^{\ell}a_i \varepsilon_i$.  Since the vector  $a$
has $0/1$ coordinates with $\ell$ ones, $\|a\|_2=\sqrt{\ell}$.
For every $0\leq k\leq \ell$ we compute
the probability $q_k$ that exactly $k$ of $\varepsilon_1, \varepsilon_2, \ldots, \varepsilon_{\ell}$ equals to one
(in that case $X=2k-\ell$). Since $S=\sum_{i=1}^{2n}\varepsilon_i=0$, in order to get $k$ ones, we have to choose
$k$ ones out of $\varepsilon_1, \varepsilon_2, \ldots, \varepsilon_{\ell}$  and $n-k$ ones out of
$\varepsilon_{\ell+1}, \varepsilon_{\ell+2}, \ldots, \varepsilon_{2n}$. This gives us
$\displaystyle{{\ell \choose k}{2n-\ell \choose n-k}}$ choices. Since
$\displaystyle{|\Omega|=\left|\left\{\varepsilon \in \{-1,1\}^{2n}\, | \,
\sum_{i=1}^{2n}\varepsilon_i=0\right\}\right|={2n \choose n}}$,
we obtain that $q_k=p_k$, i.e. $X=2(\xi-\EE\, \xi)$, where $\xi$ has hypergeometric distribution with parameters
$(2n, n, \ell)$. Therefore, Corollary 3.3 implies
\begin{align*}
\left(\EE|\xi-\EE\xi|^p\right)^{1/p}\leq \sqrt{2 \, p \, \ell}.
\end{align*}
\qed

\bigskip
We would also like to note that Proposition \ref{hypergeometric} can be proved directly.
Below we provide such a  direct proof, which gives $2$ in place of $\sqrt 2$ in front of $\displaystyle{\left(\frac{p\ell}{4}\right)^{p/2}}$. This proof is of interest as it can be extended to slightly more general case (see Remark \ref{rem4}) and can be used in another approach to the main problem (see Remark \ref{rem5}).

\bigskip

\noindent\textbf{Direct proof of the Proposition \ref{hypergeometric}. }
%
%
From Stirling's formula together with the observation
that $\sqrt{\pi n}\, {2n \choose n}/4^n$ increases, we observe that
\begin{align*}
\frac{2^{2n}}{\sqrt{2 \pi n}}\leq {2n \choose n }\leq \frac{2^{2n}}{\sqrt{ \pi n}}.
\end{align*}
Using this, we obtain
\begin{align}\label{bincoef}
\frac{{2n - \ell \choose n-k}}{{2n \choose n}}\leq \frac{{2n-\ell \choose n-\lfloor\frac{\ell}{2}\rfloor}}{{2n \choose n}}\leq \frac{2^{2n-\ell}}{\sqrt{\pi(n-\lfloor\frac{\ell}{2}\rfloor)}}\frac{\sqrt{2 \pi n}}{2^{2n}}\leq\frac{2}{2^{\ell}}\leq 1.
\end{align}
Therefore
\begin{align*}
\EE \, |\xi-\EE\, \xi|^p=\frac{1}{2^p}\sum_{k=0}^{\ell}|2k-\ell|^p\frac{{\ell \choose k}{2n-\ell \choose n-k}}{{2n \choose n}}\leq \frac{2 }{2^{\ell+p}}\sum_{k=0}^{\ell}|2k-\ell|^p {\ell \choose k}=\frac{2}{2^p} \EE|2S_{\ell}|^p,
\end{align*}
where $S_{\ell}$ is a sum of $\ell$ i.i.d. Rademacher random variables. By Khinchine inequality  (\ref{1}), we have
\[
\left(\EE|S_{\ell}|^p\right)^{1/p}\leq \sqrt p \, \sqrt{\ell}.
\]
Thus,
\[
\EE \, |\xi-\EE\, \xi|^p
\leq 2  \left(\frac{ p \, {\ell}}{4}\right)^{p/2}.
\]
$\hspace{16cm}\Box$

\bigskip

\begin{remark}\label{rem4}
The above proof can be extended to slightly larger class of hypergeometric random variables. Note that the proof works whenever $\displaystyle{{{N-\ell \choose n-k}}\bigg/{{N \choose n}}\leq 1}$. Thus, if $\displaystyle{\ell \geq N- \log_2\left[\sqrt{\pi} {N \choose n}\right] }$, then
$$
\displaystyle{\EE \, |\xi-\EE\, \xi|^p\leq 2 \left({p \, \ell}/{4}\right)^{\frac p2}}
$$
for a $(N, n, \ell)$ hypergeometric random variable $\xi$.
\end{remark}

\section{Concluding Remarks}

In this section we discuss Problem~\ref{probl}. A possible approach to this problem is to use the
concentration on the group $\Pi_N$
(endowed with the distance $\displaystyle{d_N{(\sigma, \pi)}=|\{i\, :\, \sigma(i)\neq \pi(i)\}|}$).
The following Theorem was proved by Maurey  (\cite{Maurey}, see also \cite{Schechtman}).

\begin{theorem}\label{THSCH}
\textit{Let $f:\Pi_{N}\longrightarrow\R$ be $1$-Lipschitz function. Then for all $t>0$
\begin{align}\label{7}
\mu\left(\{\sigma:|f(\sigma)-\EE f|\geq t\} \right)\leq 2e^{-t^2/(32N)}.
\end{align}}
\end{theorem}
Let us mention here, the following open question, posed by G.~Schechtman in \cite{Schechtman}: ``\textit{Is there an equivalent (with constants independent of $N$) metric on $\Pi_{N}$ for which the isoperimetric problem can be solved?''}

\bigskip
Theorem \ref{THSCH} implies the following estimate.
\begin{corollary}\label{cor5.2}
Let $a,b \in \R^N$. Let $f: \Pi_N\longrightarrow \R$ be defined by
\begin{align}
f(\sigma):=\left|\sum_{i=1}^Na_{\sigma(i)} b_i\right|.
\end{align}
Then
\begin{align}\label{10}
\left(\EE|f|^p\right)^{1/p}\leq \EE |f|+ 4\sqrt{p} \, \sqrt{N}\|a\|_{\infty}\|b\|_{\infty}.
\end{align}
\end{corollary}
\begin{proof}
It is easy to see that  $f$ is a Lipschitz function with Lipschitz constant $2\|a\|_{\infty}\|b\|_{\infty}$, indeed,
\begin{align*}
|f(\sigma)-f(\pi)|&\leq \left|\sum_{i=1}^Na_{\sigma(i)}b_i-\sum_{i=1}^{N}a_{\pi(i)}b_i\right|\\
&\leq \sum_{i=1}^{N}|b_i||a_{\sigma(i)}-a_{\pi(i)}|\leq 2\|a\|_{\infty}\|b\|_{\infty} d_{N}(\sigma, \pi).
\end{align*}
Using Theorem \ref{THSCH} and the bound $\Gamma(x)\leq x^{x-1}$ for all $ x\geq 1$ (see for example \cite{Anderson}), we obtain
\begin{align*}
\EE|f-\EE f|^p=\int_0^{\infty}\mu_{N}(|f-\EE f|^p\geq t^p)dt^p
&\leq 2p\int_0^{\infty}e^{-t^2/(32N\|a\|_{\infty}^2\|b\|_{\infty}^2)}t^{p-1}dt\\
&\leq 4^p\, \Gamma\left(\frac p2\right)N^{p/2}\|a\|_{\infty}^p\|b\|_{\infty}^p\\
&\leq 4^p\, N^{p/2}p^{p/2}\|a\|_{\infty}^p\|b\|_{\infty}^p.
\end{align*}
Thus,
\begin{align*}
\left(\EE|f|^p\right)^{1/p}\leq \EE |f|+ 4\sqrt{ p} \sqrt{N}\|a\|_{\infty}\|b\|_{\infty} \leq
\sqrt{\EE |f|^2}+ 4\sqrt{ p} \sqrt{ N}\|a\|_{\infty}\|b\|_{\infty}.
\end{align*}
\end{proof}

\bigskip

\begin{remark}\label{rem5}
In the case when $b_i=\pm 1$ with condition $\displaystyle{\sum_{i=1}^N b_i=0}$, Corollary \ref{cor5.2} gives an additional factor $\sqrt{N }$ in the upper estimate in (\ref{main}). Using the chaining argument similar to the one used in  \cite{ALPT, ALPT2, ALLPT} and Proposition \ref{hypergeometric}, the factor $\sqrt N$ can be reduced to $\sqrt{\ln N }$ (the details will be provided in \cite{S}).
\end{remark}

\begin{remark}
It would be nice to obtain the upper bound in Corollary \ref{cor5.2} with constant independent of $N$.
\end{remark}

\bigskip

\noindent
\textbf{Acknowledgement:} I would like to express my deep gratitude to my supervisor A.E.~Litvak for encouragement and helpful discussions. I am also thankful to M.~Rudelson for his valuable suggestions and very useful discussions.



\begin{thebibliography}{50}




\bibitem{ALPT}
R.~Adamczak, A.E.~Litvak, A.~Pajor, N.~Tomczak-Jaegermann,
{\it Quantitative estimates of the convergence of the empirical covariance matrix in Log-concave Ensembles,}
Journal of AMS,
\textbf{234} (2010), 535--561.

\bibitem{ALPT2}
R.~Adamczak, A.E.~Litvak, A.~Pajor, N.~Tomczak-Jaegermann,
 {\it Restricted isometry property of matrices with independent columns and neighborly polytopes by random sampling,}
  Constructive Approximation, \textbf{34} (2011), 61--88.

\bibitem{ALLPT}
R.~Adamczak, R.~Lata{\l}a, A.E.~Litvak, A.~Pajor, N.~Tomczak-Jaegermann,
{\it Tail estimates for norms of sums of log-concave random vectors,}
 Proc. London Math. Soc., to appear.

\bibitem{Anderson}
G.~D.~Anderson, S.~L.~Qiu,
{\it A monotoneity property of the gamma fuction,}
Proc. AMS,
\textbf{125} (1997), 3355--3362.




 \bibitem{Chafai}
D.~Chafai, O.~Gu$\acute{\textmd{e}}$don, G.~Lecu$\acute{\textmd{e}}$, A.~Pajor,
{\it Interaction between compressed sensing, random matrices and high dimensional geometry},
 Panoramas et Synthèses, 37 (2012).













\bibitem{Garlin}
 D.~J.~H.~Garling,
{\it Inequalities: A Journey into Linear Analysis},
Cambridge University Press, Cambridge, 2007.



\bibitem{G}
 O.~Guedon,  P.~Nayar, T.~Tkocz,
{\it Concentration inequalities and geometry of convex bodies}, Extended notes of a course, Polish Academy of Sciences of Warsaw, to appear.
(http://perso-math.univ-mlv.fr/users/guedon.olivier/listepub.html)



\bibitem{JKK}
N.~Johnson, A.W.~Kemp, S.~Kotz,
{\it Univariate discrete distributions}, Third edition,
Wiley Series in Probability and Statistics,
Wiley-Interscience [John Wiley \& Sons], Hoboken, NJ, 2005.



\bibitem{Kah} J.-P. Kahane,
{\it Some random series of functions},  Second edition.
Cambridge Studies in Advanced Mathematics, 5,
Cambridge University Press, Cambridge, 1985.






 \bibitem{LT}
J.~Lindenstrauss, L.~Tzafriri,
{\it Classical Banach Spaces I and II,}
 Springer, 1996.




\bibitem{Maurey}
B.~Maurey,
{\it Construction de suites s\'{y}metriques},
C.R. Acad. Sci. Paris S\'{e}r. A-B,
 \textbf{288} (1979), 679--681.


\bibitem{MS}
V.~D.~Milman, G.~Schechtman,
{\it Asymptotic theory of finite-dimensional normed spaces.}
With an appendix by M. Gromov. Lecture Notes in Math., 1200.
Springer-Verlag, Berlin, 1986.


\bibitem{O'Rurke}
S.~O'Rourke,
{\it A note on the Marchenko-Pastur law for a class of random matrices with dependent entries},
Electronic Communications in Probability,  \textbf{17}, no. 28 (2012), 1--13.


\bibitem{PShir}
G.~Peskir, A.~N.~Shiryaev,
\emph{ The inequalities of Khinchine and expanding sphere of their action,}
Russian MAth. Surveys
\textbf{50} 5 (1995), 849--904.

\bibitem{Schechtman}
G.~Schechtman,
{\it Concentration, results and applications},
Handbook of the geometry of Banach spaces, Vol. 2, 1603--1634,
North-Holland, Amsterdam, 2003.


\bibitem{Skala}
M.~Skala,
{\it Hypergeometric tail inequalities: ending the insanity}, preprint at
http://ansuz.sooke.bc.ca/professional/hypergeometric.pdf, 2009.

\bibitem{S}
S.~Spektor (2014). Ph.D. Thesis. University of Alberta,  Canada.


\bibitem{NTJ} N.~Tomczak-Jaegermann,
{\it Banach-Mazur distances and finite-dimensional operator ideals,}
Pitman Monographs and Surveys in Pure and Applied Mathematics, 38.
Longman Scientific \& Technical, Harlow;
copublished in the United States with John Wiley \& Sons, Inc., New York, 1989.





\end{thebibliography}
\end{document}